\documentclass[oneside]{amsart}
\usepackage{amsmath, amssymb}
\usepackage{amsthm}
\newtheorem{theorem}{Theorem}[section]
\newtheorem{proposition}[theorem]{Proposition}
\newtheorem{lemma}[theorem]{Lemma}
\newtheorem{corollary}[theorem]{Corollary}

\def\S{\mathbb{S} } 

\def\T{\mathbb{T} }

\def\R{\mathbb{R} } 
 
\def\nbd{neighborhood } 
 
\def\R{\mathbb{R} }

\title[Extended orbit properties on surfaces]
{Extended orbit properties on surfaces}
\author{Tomoo Yokoyama}
\date{\today}

\address{Department of Mathematics, Hokkaido University, 
Kita 10, Nishi 8, Kita-Ku, Sapporo, Hokkaido, 060-0810, Japan \\
}
\email{yokoyama@math.sci.hokudai.ac.jp}

\thanks{The author is partially supported 
by the JST CREST Program at Department of Mathematics,  
Hokkaido University.}

\begin{document}

\maketitle

\begin{abstract}
In this paper, we study ``demi-caract\'eristique'' and (Poisson) stability in the sense of Poincar\'e. 
Using the definitions \'a la Poincar\'e  
for $\R$-actions $v$ on compact connected surfaces, 
we show that 
``$R$-closed'' $\Rightarrow$ 
``pointwise almost periodicity (p.a.p.)''  $\Rightarrow$   
``recurrence''  $\Rightarrow$ 
non-wandering. 
Moreover, we show that 
the action 
$v$ is  ``recurrence'' with $|\mathrm{Sing}(v)| < \infty$
iff 
$v$ is regular non-wandering.  
If there are no locally dense orbits, 
then $v$ is ``p.a.p.'' 
iff 
$v$ is ``recurrence'' without ``orbits'' containing infinitely singularities.  
If 
$|\mathrm{Sing}(v)| < \infty$, 
then 
$v$ is ``$R$-closed''  
iff 
$v$ is ``p.a.p.''. 
%
\end{abstract}

\section{Introduction and preliminaries}
In the Poincar\'e celebrated paper \cite{P}  
which is an origin of dynamical systems,   
he studied surface flows. 
In the series of the relative works, 
he used the slightly different definitions from the presence notations 
(e.g. semi-characteristics, limit cycles). 
On the other hand, 
the following fact for topological dynamics on compact metrizable spaces is known: 
$ 
\text{$R$-closed} \subsetneq
\text{p.a.p.} \subsetneq
\text{recurrent} \subsetneq
\text{non-wandering}. 
$
In this paper, 
we study a surface flow using the notations of Poincar\'e. 
In particular, we study ``demi-caract\'eristique'' and (Poisson) stability in the sense of Poincar\'e  
(We call these extended positive orbits and extended recurrence).   
Precisely, 
we show the following relation for $\R$-actions on compact surfaces:  
$$ 
\text{extended $R$-closed} \subsetneq
\text{extended p.a.p.} \subsetneq
\text{extended  recurrent} \subsetneq
\text{non-wandering}. 
$$
Moreover, we show that 
the $\R$-action 
$v$ on a compact surface $S$
is   extended recurrence  
with at most finitely many singularities 
if and only if 
$v$ is regular non-wandering.  
If $v$ has no locally dense orbits, 
then 
$v$ is extended recurrence with 
$| \mathrm{Sing}(v) \cap O_\mathrm{{ex}}(x)| < \infty$ for each point $x \in S$  
if and only if 
$v$ is  extended p.a.p..  
If 
$|\mathrm{Sing}(v)| < \infty$, 
then 
$v$ is extended $R$-closed   
if and only if 
$v$ is extended p.a.p..

Recall ``demi-caract\'eristique'' in the sense of Poincar\'e. 
Let $v$ be an $\R$-action on a surface $S$. 
For a singular point $x$ of $S$, 
we call that 
$x$ is a (topological) saddle for a continuous $\R$-action
if 
there is a \nbd of $x$ which is 
locally homeomorphic to a \nbd of a saddle for a $C^1$ $\R$-action. 
For a point $x$ of $S$, 
define $O^+_i(x)$ as follows: 
$O^+_0 := O^+(x)$,  
$$O^+_{i + 1}(x) := O^+_i(x) \cup \cup_{x' \in O^+_i(x)} \{ W^u(\omega(x')) \mid  \omega(x') : \text{ 
saddle } \}$$ 
for any successor ordinal $i$, 
and 
$O^+_{\mu} := \cup_{\mu > \nu}O^+_{\nu}$ 
for any limit ordinal $\nu$. 
Here $W^u(y) := \{ z \in S \mid \alpha(z) = \{ y \} \}$. 
Put $O^+_\mathrm{{ex}}(x) := \cup \{ O^+_{\nu}(x) \mid {\nu : \text{ordinal}} \}$ 
is called the extended positive orbit of $x$. 
Note Poincar\'e called this demi-caract\'eristique. 
Similarly, we defined the extended negative orbit $O^-_\mathrm{{ex}}(x)$ and so 
define 
the extended orbit $O_\mathrm{{ex}}(x) := O^+_\mathrm{{ex}}(x) \cup O^-_\mathrm{{ex}}(x)$. 
Note that generally 
$O_\mathrm{{ex}}(x) \neq O_\mathrm{{ex}}(y)$ 
for a point  $x \in S$ and for a point $y \in O_\mathrm{{ex}}(x)$. 
In fact, 
the binary relation $\{ (x, y) \mid y \in O_\mathrm{{ex}}(x) \}$ 
is reflexive and symmetric but need not transitive.  
If the extended orbit $O_\mathrm{{ex}}(x)$ is not a single point 
but a compact subset, 
then it is called an extended periodic orbit. 
Notice that 
the positive prolongations and  
the extend positive orbits are independent. 
%
%
%
A  closed subset $\gamma$ of an extended orbit is called 
a limit cycle in the sense of Poincar\'e 
or an extended limit cycle 
if 
$\gamma$ is 
not a singleton but 
a union of simple closed curves 
and 
there is a point $x$ of $S - \gamma$ 
such that 
$\gamma$ is 
either the omega limit  set $\omega(x)$ 
or the alpha limit set $\alpha(x)$ of $x$. 
A point $x$ of $S$ is extended positive recurrent
if either $x$ is positive recurrent 
or 
$x \in \overline{O^+_\mathrm{{ex}}(x) - O^+(x)}$.  
Similarly, we define ``extended negative recurrent''. 
A point $x$ of $S$ is extended recurrent
if $x$ is extended positive recurrent 
and extended negative recurrent. 
%
The $\R$-action $v$ is said to be 
extended 
recurrent if 
so is each point of $S$.  
By definitions, 
recurrence implies extended recurrence. 
Moreover 
each flow on a compact surface which consists of 
closed orbits and 
at least one 
 saddle connections 
is not recurrent but extended recurrent. 
In addition, 
there is a flow which is  
not  extended recurrent but non-wandering. 
Indeed,  consider 
the unit sphere $\S^2 \subset \R^3$ 
and define 
$v_t: \S^2 \to \S^2$ with $\mathrm{Fix}(v) = \{(1, 0, 0) \}, \{ (0, 0, \pm 1) \}$ 
such that  
the regular orbits consists of 
$(\S^1 \times \{0 \}) - \{(1, 0, 0) \}$ 
and of 
circles each of which is 
the intersection of $\S^2 \cap (\R^2 \times \{z\})$ for some $z \neq 0 \in (-1,1)$.

\section{Extended recurrence and Non-wandering property}

From now on, 
let $v$ be an  $\R$-action on a compact surface $S$.  
Denote by $\mathrm{LD}$ (resp. $\mathrm{E}$)
the union of locally dense (resp. exceptional) orbits of $v$. 
Recall that 
an orbit $O$ is proper if 
$\overline{O} - O$ is closed, 
is locally dense if 
$\mathrm{int}\overline{O} \neq \emptyset$, 
and is exceptional if 
$O$ is neither proper nor locally dense 
 and 
that 
a point is proper (resp. locally dense, exceptional) 
if so is it. 
Let $\mathrm{P}$ be the union of points whose orbits are not closed but proper.  
%
Recall the following fundamental fact. 

\begin{lemma}\label{lem0a} 
The set of saddles are countable. 
\end{lemma}

\begin{proof} 
By the definition of saddles, 
each saddle has a \nbd which contains no other saddles. 
Since $S$ is second countable, 
the set of saddles can be enumerated 
and so is countable. 
\end{proof}

The above proof also shows that 
an extended orbit $O_\mathrm{{ex}}(x)$ for a point $x$ of $S$ 
contains at most countably many saddles and 
is $O_\mathrm{\aleph_1}(x)$. 
Now we show the useful tools. 

\begin{lemma}\label{lem0} 
Each extended periodic orbit $O$
 consists of 
finitely many proper orbits and saddles. 
\end{lemma}

\begin{proof} 
We show that 
$O$ contains at most finitely many saddles.  
Otherwise 
$O$ contains infinitely many saddles.  
The definition of saddles implies that 
$O$ contains a singularity which is not a saddle, 
which contradicts to the definition of extended orbits. 
Then $O$ contains at most finitely many distinct orbits. 
If $O$ contains 
either locally dense orbits 
or 
exceptional orbits, 
then 
the closedness of $O$ implies that 
$O$ contains at least uncountably many orbits, 
which contradicts to the finiteness. 
Thus $O$ consists of 
finitely many proper orbits and saddles. 
\end{proof}

\begin{lemma}\label{lem00} 
If 
there are extended limit cycles, 
there is a wandering point $x \ in \mathrm{P}$ 
such that 
$O(x) = O_\mathrm{{ex}}(x)$  
\end{lemma}

\begin{proof} 
Suppose that 
there is an extended limit cycle $C$. 
We may assume that 
there is a point whose 
omega limit set is $C$. 
Then 
there are uncountably many  proper orbits 
each of whose omega limit set 
is $C$. 
Since the set of saddles are countable, 
there is a proper orbit $O$ whose extend orbit is coincident with itself 
such that $\omega(O) = C$. 
This implies that 
each point of $O$ is wandering. 
\end{proof}

\begin{lemma}\label{lem1} 
If 
$\mathrm{P}$ consists of at most finitely many orbits, 
then 
$v$ is non-wandering. 
\end{lemma}

\begin{proof} 
It suffices to show that 
each point $x \in \mathrm{P}$ is non-wandering. 
Indeed, 
By the flow box theorem, 
we have 
$x \in \overline{\mathrm{E} \sqcup \mathrm{LD} \sqcup \mathrm{Per}(v)}$. 
Since each point of $\mathrm{E} \sqcup \mathrm{LD} \sqcup \mathrm{Per}(v)$ 
is either positive or negative recurrent, 
we have that 
$x$ is non-wandering and so 
$v$ is non-wandering. 
\end{proof}

\begin{lemma}\label{lem023} 
Suppose that $v$ is extended recurrent. 
For any point $x$ which is regular or is a saddle, 
there is a \nbd $U$ such that $U - O_\mathrm{{ex}}(x)$ contains no singularities. 
\end{lemma}

\begin{proof} 
If $O_\mathrm{{ex}}(x)$ contains no saddles, 
then it contains no singularities 
and so  
the flow box theorem implies the assertion. 
Thus we may assume that 
$O_\mathrm{{ex}}(x)$ contains saddles points. 
By the definition of extended orbits, 
we obtain that 
$O_\mathrm{{ex}}(x) \cap \mathrm{Sing}(v)$ consists of saddles points. 
Since each saddle $p$ has a \nbd $U_p$ such that 
$U_p - \{ p \}$ consists of regular points. 
By the flow box theorem, 
there is a \nbd $U$ of $O_\mathrm{{ex}}(x)$ such that 
$U - O_\mathrm{{ex}}(x)$  contains no singularities. 
\end{proof}

The extended recurrence implies 
the (usual) non-wandering property. 

\begin{lemma}\label{lem2} 
If $v$ is  extended recurrent, 
then 
$v$ is non-wandering. 
\end{lemma}

\begin{proof} 
Note 
$S - \mathrm{P} = 
\mathrm{E} \sqcup 
\mathrm{LD} \sqcup 
\mathrm{Per}(v) \sqcup 
\mathrm{Sing}(v)$. 
If $\mathrm{int}\mathrm{P} = \emptyset$, then 
the closedness of $\mathrm{Sing}(v)$ implies that 
$\overline{\mathrm{E} \sqcup 
\mathrm{LD} \sqcup 
\mathrm{Per}(v)} \supset \mathrm{P}$ 
and so 
$v$ is non-wandering. 
Thus 
it suffices to show 
$\mathrm{int}\mathrm{P} = \emptyset$. 
Indeed, 
recall that 
the set of saddles are countable. 
For any $x \in \mathrm{P}$, 
the extended recurrence implies that 
the omega (resp. alpha) limit set of $x$ is a saddle. 
Therefore $\mathrm{P}$ consists of countable orbits. 
Since $S$ is a Baire space, 
we have that 
$\mathrm{int}\mathrm{P} = \emptyset$. 
\end{proof}

Recall that 
a continuous $\R$-action $v$ is regular if 
each singularity of $v$ is locally homeomorphic to 
a non-degenerated singularity of a $C^1$ vector field. 
Note 
the non-wandering flow $v$ has no no exceptional orbits 
such that 
$\overline{\mathrm{LD} \sqcup \mathrm{Per}(v)} \supseteq  S- \mathrm{Sing}(v)$, 
by 
Lemma 2.1 \cite{Y}.

\begin{proposition}\label{lem23} 
Suppose that $v$ is non-wandering. 
Then 
$v$ is regular if and only if 
$v$ is extended recurrent and 
has finitely many singularities. 
Moreover,  
if $v$ is regular, then 
either 
$O_\mathrm{{ex}}(x)$ is closed 
or 
both 
$O^+_\mathrm{{ex}}(x)$ 
and 
$O^-_\mathrm{{ex}}(x)$ are locally dense 
for any $x \in S$. 
\end{proposition}

\begin{proof} 
Suppose that $v$ is regular. 
The regularity implies that 
each singularity is either 
a center, 
a saddle, 
a sink, 
or 
a source. 
By the non-wandering property, 
we have that 
there are no limit cycles 
and 
that 
each singularity is 
either 
a center or a saddle. 
By the regularity, 
the set of saddles is finite. 
Since the omega (resp. alpha) limit set of each non-closed proper orbit is 
a saddle, 
we have that 
the set of non-closed proper orbits are finite. 
It suffices to show that 
each point 
$x \in S$ whose extended orbit is not closed but proper is extended recurrent. 
Indeed, 
we may assume that 
there are no $y \in O^+_\mathrm{{ex}}(x) - O^+(x)$ 
such that 
$x \notin {O^+_\mathrm{{ex}}(y)}$. 
%
Since each saddle has two local (un)stable manifolds, 
both  
$O^+_\mathrm{{ex}}(x)$ and 
$O^-_\mathrm{{ex}}(x)$ are not closed. 
Since the union of non-closed proper orbits is finite and 
since each non-closed proper orbit is a saddle connection,  
we have that 
$O^+_\mathrm{{ex}}(x)$ 
(resp. $O^-_\mathrm{{ex}}(x)$) contains 
a locally dense orbit. 
Let $U$ be a \nbd of $O_\mathrm{{ex}}(x)$ 
such that 
$U - O_\mathrm{{ex}}(x)$ contains no singularities. 
By the finiteness of saddles, 
there is an arbitrary thin connected open subset $U_x \subseteq U$ which is 
disjoint from the union of heteroclinic connections in $O^+_\mathrm{{ex}}(x)$ 
and whose closure contains a curve $C^+$ in $O^+_\mathrm{{ex}}(x)$ from $x$ 
to a point in a locally dense orbit $O^+$ 
such that the orientations of $C^+$ and $O^+_\mathrm{{ex}}(x)$ are same. 
By locally density, 
we have 
$C^+ \subset \overline{U_x \cap \overline{O^+}}$
and so 
that $x \in \overline{O^+} \subseteq \overline{O^+_\mathrm{{ex}}(x) - O^+(x)}$. 
By the symmetry, 
this implies that $x$ is extended recurrent 
and so 
$v$ is extended recurrent. 

Conversely, 
suppose that 
$v$ is extended recurrent and has finitely many singularities.
By the finiteness of singularities, 
we have 
$\overline{\mathrm{Per}(v) \sqcup \mathrm{LD}} = S$. 
Since 
each connected component $C$ of 
the boundary of $\mathrm{Per}(v)$ consists of 
proper orbits and finitely many singularities,  
by the extended recurrence, 
we have that 
$C$ is 
either a center 
or a closed extended orbit 
and so 
each singularity contained in $C$ is a center or a saddle. 
On the other hand, 
the boundary of $\mathrm{LD}$ consists of 
proper orbits and finitely many singularities. 
The extended recurrence implies that 
each singularity in the boundary is a saddle. 
Thus $v$ is regular. 
\end{proof}

Now we describe an $\R$-action 
which has non-closed extended orbits and 
which is not recurrent but extended recurrent. 
Consider 
an irrational rotation on $\T^2$ 
and a rational rotation on $\T^2$. 
Removing a point from each torus, 
paste the metric completions of them such that 
the intersection is a circle which consists of 
two saddles and two heteroclinic connections. 
Then we obtain an extended recurrent $\R$-action on 
a closed oriented surface with genus $2$ which 
is not recurrent and has non-closed extended orbits. 
Notice that 
this example shows also that 
the extended orbits are different from 
chain recurrent components. 
We obtain the following dichotomy for extended recurrent $\R$-actions.

\begin{lemma}\label{prop11} 
Suppose that $v$ is extended recurrent. 
For any point $x$ 
whose extended orbit is not closed, 
either 
there is a singular point in $\overline{O_\mathrm{{ex}}(x)}$ which is not a saddle 
or 
there is a locally dense orbit $O$ such that 
$O_\mathrm{{ex}}(x) \cap \overline{O} \neq \emptyset$. 
\end{lemma}

\begin{proof} 
Since extended recurrence implies non-wandering property, 
there are no exceptional orbits and 
$\overline{\mathrm{LD} \sqcup \mathrm{Per}(v)} \supseteq  S- \mathrm{Sing}(v)$.
Suppose that 
$O_\mathrm{{ex}}(x) \cap \overline{\mathrm{LD}} = \emptyset$. 
Then 
$O_\mathrm{{ex}}(x)$ consists of proper orbits and saddles. 
The extended recurrence implies that 
the omega (resp. alpha) limit set of each proper orbit in $O_\mathrm{{ex}}(x)$ 
is a saddle. 
The non-closedness of $O_\mathrm{{ex}}(x)$ 
implies that 
$O_\mathrm{{ex}}(x)$ contains infinitely many saddles. 
Since saddles are isolated, 
a convergence point of saddles is a singular point which is not a saddle. 
This singularity is desired. 
Suppose that 
$O_\mathrm{{ex}}(x) \cap \overline{\mathrm{LD}} \neq \emptyset$.  
By the Ma\v \i er Theorem  \cite{M}, 
the set of closures of locally dense orbits is finite 
and so there is a locally dense orbit $O$ such that 
$O_\mathrm{{ex}}(x) \cap \overline{O} \neq \emptyset$. 
\end{proof}

Note that 
there is an extended recurrent $\R$-action with 
a non-closed proper extended orbit with infinitely many saddles on a disk. 
Indeed, 
let $S := \{ (x, y) \in \R^2 \mid x^2 + y^2 \leq 2 \}$. 
Consider circles $S_n := \{ (x, y) \in \R^2 \mid (x -3/2^n)^2 + y^2 = 2^{-n} \}$ 
for each $n \geq 2 \in \mathbb{Z}$. 
Let $O := \cup_{n \geq 2} S_n$.  
Define an $\R$-action $v$ with an extended orbit $O$  
such that 
the origin is a fixed point which is not a saddle, 
the outside of $\overline{O}$ consists of periodic orbits,   
and 
each open disk bounded by $S_n$ 
is a center disk. 
Then $v$ is extended recurrent and 
has one non-closed proper extended orbit $O$ 
with infinitely many saddles.

\section{Pointwise almost periodicity}

We define extended versions of pointwise almost periodicity. 
An $\R$-action $v$ on a topological space $X$ is said to be 
extended pointwise almost periodic (extended p.a.p.)
if 
the set $\{\overline{O_\mathrm{{ex}}(x)} \mid x \in X \}$ 
of closures of 
extended orbits is a decomposition of $X$.

\begin{lemma}\label{lem21} 
If $v$ is an extended p.a.p. $\R$-action on a compact surface, 
then 
$v$ is extended recurrent  
and $| \mathrm{Sing}(v) \cap \overline{O_\mathrm{{ex}}(x)}| < \infty$ for each $x \in S$. 
\end{lemma}

\begin{proof} 
Fix each regular point $y \in S$ 
such that 
$O_\mathrm{{ex}}(y)$ contains singularities. 
By definition of extended orbits, 
the singularities in $O_\mathrm{{ex}}(y)$ 
are saddles. 
If 
$\overline{O_\mathrm{{ex}}(y)}$ contains 
a singular point $p$ which is not a saddle, then  
$\overline{O_\mathrm{{ex}}(p)} 
= \{ p \}  \subsetneq \overline{O_\mathrm{{ex}}(y)}$
which contradicts to the extended p.a.p..
Thus 
$\overline{O_\mathrm{{ex}}(y)} \cap \mathrm{Sing}(v)$ 
consists of finitely many saddles.  
%
%
%
%
%
%
%
Since the set of saddles is countable, 
the extended p.a.p. property implies that 
$\mathrm{P}$ consists of countably many orbits. 
The flow box theorem implies that 
$\mathrm{P} \subset \overline{\mathrm{E} \sqcup \mathrm{LD} \sqcup \mathrm{Per}(v)}$ 
and so that 
$v$ is non-wandering. 
Fix any point $x \in \mathrm{P}$ whose extended orbit is 
not closed. 
We show that 
either omega or alpha limit set  $L$ of $x$ 
is a saddle. 
Otherwise 
there is a point $z \in \overline{O_\mathrm{{ex}}(x)}$ 
whose omega (resp. alpha) limit set is not a saddle. 
Then 
$O(z) = O_\mathrm{{ex}}(z)$ 
and 
$O(z) \cap \omega(z) = \emptyset$.  
On the other hand, 
$O_\mathrm{{ex}}(x) \subseteq \omega(z)$ 
and so 
$\overline{O_\mathrm{{ex}}(x)} \subseteq \omega(z)$. 
Since $z \notin  \omega(z)$, 
we have 
$z \notin \overline{O_\mathrm{{ex}}(x)}$, 
which contradicts to the choice of $z$.  
%
%
%
%
%
%
%
%
%
%
%
Since $| \mathrm{Sing}(v) \cap \overline{O_\mathrm{{ex}}(x)}| < \infty$, 
each of  $O^+_\mathrm{{ex}}(x)$ and 
$O^-_\mathrm{{ex}}(x)$
contains locally dense orbits. 
By symmetry, 
it suffices to show that 
$x \in \overline{O^+_\mathrm{{ex}}(x) - O^+(x)}$.  
Indeed, 
we may assume that 
there is no point $y \in O^+_\mathrm{{ex}}(x) - O^+(x)$ with 
$x \in O^+_\mathrm{{ex}}(y)$. 
Since 
$\mathrm{Sing}(v) \cap \overline{O_\mathrm{{ex}}(x)}$ consists of finitely many saddles, 
there is a thin connected open subset $U_x$ without singularities 
whose closure contains a curve in $O^+_\mathrm{{ex}}(x)$ from $x$ 
to a point $w \in \mathrm{LD}$
such 
that 
the orientations of the curve and $O^+_\mathrm{{ex}}(x)$ are compatible. 
Then 
$x \in 
\overline{O^+(w)} 
\subseteq 
\overline{O^+_\mathrm{{ex}}(x) - O^+(x)}$. 
%
%
\end{proof}

In the case without locally dense orbits, 
the following equivalence holds.

\begin{proposition}\label{prop32} 
Suppose that $v$ is a non-identical $\R$-action 
 without locally dense orbits 
on a compact surface $S$. 
The following are equivalent: 
\\
1) 
$v$ is extended p.a.p.. 
\\
2) 
$v$ is extended recurrent and $| \mathrm{Sing}(v) \cap \overline{O_\mathrm{{ex}}(x)}| 
< \infty$ for each $x \in S$. 
\\ 
3)
$v$ consists of closed extended orbits. 
\end{proposition}

\begin{proof} 
Obviously 
$3) \Rightarrow 1)$. 
By Lemma \ref{lem21}, 
we have that 
$1) \Rightarrow 2)$. 
Suppose that 
$2)$ holds. 
Moreover suppose that 
there is a non-closed extended orbit $O_\mathrm{{ex}}(x)$. 
By Lemma \ref{prop11}, 
there is a singularity $z$ in $\overline{O_\mathrm{{ex}}(x)}$
which is not a saddle.  
The extended recurrence implies that 
$\{ z \} \neq \alpha(y)$ 
and 
$\{ z \} \neq \omega(y)$ 
for any $y \neq z \in S$.  
This contradicts to 
the Ura-Kimura-Bhatia theorem (cf. Theorem 1.6 \cite{B}). 
Thus $v$ consists of closed extended orbits. 
\end{proof}

%
%
Note that 
there is an $\R$-action on a connected closed surface which is 
not extended p.a.p. but extended recurrent 
and whose singularities consists of two saddles. 
Indeed, 
consider 
two irrational rotations on $\T^2$.  
Let $T_1, T_2$ be the metric completions of 
the resulting surfaces by removing one point from each torus. 
Then $T_i$ is homeomorphic to a torus minus an open disk.  
Paste them such that 
the resulting surface $S = T_1 \cup T_2$ is a closed orientable surface with genus $2$ 
and that 
the intersection $T_1 \cap T_2$  
is a circle which consists of 
two saddles and two heteroclinic connections. 
Let $v$ be the resulting $\R$-action on $S$. 
The extended orbit closure of each point of $(\mathrm{int} T_i) \setminus O_\mathrm{{ex}}(x)$ 
for a point $x \in T_1 \cap T_2$ 
is $T_i$, 
and 
the extended orbit closure of each point 
$x \in O_\mathrm{{ex}}(x_1) \cup O_\mathrm{{ex}}(x_2)$ 
is $S$, where any $x_i \in T_i$.  
Then $v$ is not extended p.a.p.. 
The extended recurrence is obviously. 

%

\section{Extended $R$-closedness}

%
%
Define extended versions of $R$-closedness. 
An $\R$-action $v$ on a compact surface $S$ is said to be 
extended $R$-closed if 
$R_{\text{ex}} := \{ (x, y ) \mid  y \in \overline{O_\mathrm{{ex}}(x)} \}$ is closed. 
%
%
%

\begin{lemma}\label{lem24} 
If $v$ is extended $R$-closed, 
then $v$ is extended p.a.p..  
%
\end{lemma}

\begin{proof} 
First we show that 
$R_{\text{ex}}$ is symmetric. 
Indeed, 
the definition of extended orbits implies that 
$\{ (x, y ) \mid  y \in {O_\mathrm{{ex}}(x)} \}$ is symmetric. 
For any $y \in \overline{O_\mathrm{{ex}}(x)}$, 
let $( y_n)$ be a sequence of points in ${O_\mathrm{{ex}}(x)}$ 
converging to $y$. 
Since $x \in {O_\mathrm{{ex}}(y_n)}$, 
we have 
$(y_n, x) \in R_{\text{ex}}$. 
The extended $R$-closedness 
implies  
$(y,  x) \in R_{\text{ex}}$ 
and so 
$x \in \overline{O_\mathrm{{ex}}(y)}$.  
%
%
The closure of 
each extended orbit contains at most finitely many singularities  
and 
either $\omega(x)$ or 
$\alpha(x)$ is a saddle 
for any $x \in \mathrm{P}$. 
Hence $\mathrm{P}$ consists of at most countably many orbits.  
By the flow box theorem, 
we obtain that 
$\mathrm{P} \subset \overline{\mathrm{LD} \sqcup \mathrm{Per}(v) \sqcup \mathrm{E}}$. 
This implies that $v$ is non-wandering.  
%
Fix any point $x \in S$. 
By symmetry, 
it suffices to show that 
$\overline{O_\mathrm{ex}(y)} \subseteq \overline{O_\mathrm{{ex}}(x)}$ 
for any $y \in O^+_\mathrm{{ex}}(x)$.  
We may assume that ${O_\mathrm{{ex}}(x)}$ is not closed.
Then there is a point $z \in O^+_\mathrm{{ex}}(y)$ whose orbit is locally dense. 
Since the set of recurrent points is dense, 
there is a recurrent point $w \in \mathrm{int} \overline{O^+(z)}$ whose orbit is locally dense. 
For any $z' \in O^-_\mathrm{{ex}}(z)$, 
we have $w \in \overline{O_\mathrm{{ex}}(z')}$ 
and so 
$z' \in \overline{O_\mathrm{{ex}}(w)} 
= \overline{O^+(w)} 
\subseteq \overline{O^+(z)} 
\subseteq \overline{O_\mathrm{{ex}}(x)}$.  
Then 
$\overline{O^-_\mathrm{{ex}}(y)} 
\subseteq \overline{O^-_\mathrm{{ex}}(z)} 
\subseteq \overline{O_\mathrm{{ex}}(x)}$ 
and so 
$\overline{O_\mathrm{{ex}}(y)} 
\subseteq \overline{O_\mathrm{{ex}}(x)}$. 
\end{proof}

For a singular point $x$, 
we call that 
$x$ is an extended center if 
there is a \nbd $U$ of $x$ 
such that $U - \{ x \}$ consists of extended periodic orbits and centers.

\begin{lemma}\label{lem034} 
Suppose that 
$v$ is non-identical extended $R$-closed and 
$S$ is connected. 
Then 
$\overline{\mathrm{LD}} \cap \mathrm{Sing}(v)$ 
is finite 
and 
all singularities are saddles and extended centers.  
\end{lemma}

\begin{proof} 
Since $v$ is non-wandering, 
there are no exceptional orbits 
and 
$\overline{\mathrm{Per}(v) \cup \mathrm{LD}} \supseteq S -\mathrm{Sing}(v)$. 
By the extended $R$-closedness, 
we have that 
each connected component of the boundary of 
${\mathrm{Per}(v)}$ 
(resp. $\mathrm{LD}$)
is contained in one extended orbit 
and so that 
$\overline{\mathrm{Per}(v)} \cap \mathrm{Sing}(v)$ consists of saddles and extended centers.  
The extended recurrence also 
implies that 
$\overline{\mathrm{LD}} \cap \mathrm{Sing}(v)$ consists of saddles.  
Since each saddle is isolated, 
we have that 
$\overline{\mathrm{LD}} \cap \mathrm{Sing}(v)$ is finite. 
By Lemma \ref{lem023}, 
$\overline{\mathrm{Per}(v) \cup \mathrm{LD}}$ is clopen 
and so $S = \overline{\mathrm{Per}(v) \cup \mathrm{LD}}$. 
Thus 
each singularity is  either 
a saddle or an extended center. 
\end{proof}

There is an extended $R$-closed flow with infinitely many saddles. 
Indeed, consider a center disk and a converging sequence of periodic orbits  
to the center. 
Replacing the periodic orbits 
by homoclinic saddle connections with center disks, 
we obtain an extended center disk with infinitely many saddles. 
By doubling this disk, we obtain an extended $R$-closed flow 
on $\S^2$ with two extended centers and with infinitely many saddles. 
Consider the case with finitely many singularities. 

\begin{proposition}\label{lem42} 
Suppose 
$|\mathrm{Sing}(v)|< \infty$. 
Then 
$v$ is extended $R$-closed 
if and only if 
$v$ is extended p.a.p..  
\end{proposition}

\begin{proof} 
It suffices to show the ``if'' part.
Suppose that $v$ is extended p.a.p..  
By Proposition  \ref{lem23}, 
we have that 
each singularity is regular and so 
is a center or a saddle. 
By the Ma\v \i er Theorem  \cite{M}, 
the set of closures of locally dense orbits is finite.  
Then 
$S - \overline{LD} \subseteq \mathrm{int}\overline{\mathrm{Per}(v)}$ 
consists of 
periodic orbits, 
finitely many centers,  and 
finitely many closed extended orbits. 
By the extended p.a.p. property, 
we obtain that 
$\overline{LD}$ consists of finitely many minimal sets with respect to extended orbits. 
For any connected component $C$ of $\overline{LD}$, 
there is a \nbd $U$ of $C$ with $U - C \subseteq \mathrm{Per}(v)$.   
Consider the quotient map $\pi: S \to S/\overline{O_\mathrm{ex}}$ of closures of 
extended orbits. 
Then $\overline{LD}$ is the inverse image of a finite subset of $S/\overline{O_\mathrm{ex}}$ 
and $\pi(S - \overline{LD})$ is a forest (i.e. a disjoint union of trees). 
Then $S/\overline{O_\mathrm{ex}}$ is Hausdorff. 
By Lemma 2.3 \cite{Y2}, 
we have that 
$v$ is extended $R$-closed. 
\end{proof}

%
%


The finiteness 
and 
the non-existence of locally dense orbits 
imply the following corollary. 

%

\begin{corollary} 
Suppose that $v$ is a non-identical $\R$-action with finitely many singularities 
on  a compact surface with genus $0$. 
The following are equivalent: 
\\
1) 
$v$ is extended $R$-closed. 
\\ 
2) 
$v$ is extended p.a.p.. 
\\ 
3) 
$v$ is extended recurrent.   
\\
4) 
$v$ is regular non-wandering. 
\\
5)
$v$ consists of closed extended orbits. 
\end{corollary}

\section{An extended non-wandering}

Naturally, 
we can define 
extended non-wandering property as others. 
It's easy to see that 
extended 
non-wandering property 
and 
(usual) 
non-wandering property 
are equivalent 
if the set of singularities are finite. 
The author don't know 
whether 
these notions 
are same or not 
in general.

\section{A note for a more generalization of orbits}

Let $F$ be a compact invariant set of $v$. 
Then 
$F$ is said to be isolated (from minimal sets) if  
there exists a neighborhood $U$ of $F$ such that any minimal set contained in $U$
is a subset of $F$.
$F$ is called a saddle set if there exists a \nbd $U$ of
$F$ such that $\overline{G}_{U}\cap F\neq\phi$, 
where 
$G_{U} := \{ x \in \overline{U} - F \mid  O^{+}(x) \not\subseteq \overline{U}, O^{-}(x) \not\subseteq \overline{U}\}$. 
In the definition of extended orbits of $x$, 
if we replace saddles with isolated saddle sets, 
then we call that 
the resulting extended orbits are 
generalized extended orbits, 
denoted by 
$O^-_\mathrm{{ge}}(x)$, 
$O^+_\mathrm{{ge}}(x)$, 
$O_\mathrm{{ge}}(x)$. 
Also we define some ``generalized'' notation by replacing saddles with isolated saddle sets. 
By the definitions, 
we notice that 
extended recurrence implies 
generalized recurrence. 
Then 
one can show the generalized version of 
Lemma \ref{lem023} 
in the similar fashion 
if one replaces saddles (resp. ex) with isolated saddle sets (resp. ge). 
However, 
this generalization does not imply 
 the generalized version of 
Lemma \ref{lem2}, \ref{prop11}. 
%
Moreover 
non-wandering property and 
generalized recurrence are independent. 
In fact, 
the following example is a vector field on $\S^2$ which is not non-wandering 
but generalized recurrent. 
Let 
$D = \{ (x, y) \in \R^2 \mid x^2 + y^2 \leq 1 \}$,
$D_+ = \{ (x, y) \in D \mid x > 0 \}$,
$D_- = \{ (x, y) \in D \mid x < 0 \}$,
$p_+ := (0, 1)$, 
$p_- := (0, -1)$,  
and let $v$ be a flow on $D$ 
such that 
$\mathrm{Fix}(v) = \{ (0, y) \in D \}$ 
and that 
$\alpha(p) = p_{\mp}$ 
and 
$\omega(p) = p_{\pm}$ 
for each point $p \in D_{\pm}$. 
Pasting an open center disk, 
we obtain a flow $v'$ on $\S^2$ 
whose fixed point set consists of 
$\mathrm{Fix}(v)$ and the center. 
Then $p_-, p_+$ are isolated saddle sets 
and so 
$D - \mathrm{Fix}(v) 
\subset   
O^{-}_{\mathrm{ge}}(y) 
= O^{+}_{\mathrm{ge}}(y)$ 
for any $y \in D - \mathrm{Fix}(v)$. 
This implies that 
$v'$ is generalized recurrent. 
On the other hand, 
the following example is a vector field on $\S^2$ 
which is not 
generalized recurrent 
but 
non-wandering. 
Define 
a vector field 
whose orbits consists of  
$\{ (1/n, 0) \}$, 
$\{ 1/n \} \times (\T^1 - \{ 0 \})$,   
and 
$\{ x \} \times \T^1$ for 
$n \in \mathbb{Z}_{>0}$ and for 
$x  \in \T^1 - \{ 1/m \mid m \in \mathbb{Z}_{>0} \}$. 
%
%
Because $\alpha (p) = \omega (p) = \{(0, 0) \}$ is a saddle set 
but not isolated for any $p \in \{ 0 \} \times (\T^1 - \{ 0 \})$
and so $O_\mathrm{{ex}}(p) = \{ 0 \} \times  (\T^1 - \{ 0 \})$ 
is not generalised recurrent.

\end{document}